\documentclass[oneside,reqno]{amsart}

\usepackage{amsmath}
\usepackage{amssymb}
\usepackage{graphicx}
\usepackage{amssymb}
\usepackage{amsbsy}
\usepackage[utf8]{inputenc}
\usepackage{bm}

\def\O{\Omega}
\def\bM{{\bf M}}
\def\bK{{\bf K}}
\def\cM{\mathcal{M}}
\def\R{\mathbb{R}}

\newcommand{\ga}{\alpha}
\newcommand{\gb}{\beta}
\newcommand{\gs}{\sigma}
\newcommand{\gve}{\varepsilon}
\newcommand{\bq}{{\bf q}}
\newcommand{\bL}{{\bf L}}
\newcommand{\bs}{{\bm \sigma}}
\newcommand{\bt}{{\bm \tau}}

\newtheorem{thm}{Theorem}[section]

\newtheorem{lemma}[thm]{Lemma}
\newtheorem{corollary}[thm]{Corollary}

\newcommand{\la}{\langle}
\newcommand{\ra}{\rangle}

\begin{document}
\title[A weighted setting for the Poisson problem with singular sources]{A weighted
setting for the numerical approximation of the Poisson problem with singular sources}

\author{Irene Drelichman}
\address{IMAS (UBA-CONICET) \\
Facultad de Ciencias Exactas y Naturales\\
Universidad de Buenos Aires\\
Ciudad Universitaria\\
1428 Buenos Aires\\
Argentina}
\email{irene@drelichman.com}

\author{Ricardo G. Dur\'an}
\address{IMAS (UBA-CONICET) and Departamento de Matem\'atica\\
Facultad de Ciencias Exactas y Naturales\\
Universidad de Buenos Aires\\
Ciudad Universitaria\\
1428 Buenos Aires\\
Argentina}
\email{rduran@dm.uba.ar}

\author{Ignacio Ojea}
\address{Departamento de Matem\'atica\\
Facultad de Ciencias Exactas y Naturales\\
Universidad de Buenos Aires\\
Ciudad Universitaria\\
1428 Buenos Aires\\
Argentina}
\email{iojea@dm.uba.ar}

\thanks{Supported by ANPCyT under grant PICT 2014-1771, by CONICET under grant 11220130100006CO  and by Universidad de Buenos Aires under grant 20020120100050BA. The first and second authors are members of
CONICET, Argentina.}



\begin{abstract}

We consider the approximation of Poisson type problems where the source is given by a singular measure
and the domain is a convex polygonal or polyhedral
domain.
First, we prove the well-posedness of the Poisson problem
when the source belongs to the dual of a weighted Sobolev space where the weight
belongs to the Muckenhoupt class. Second, we prove the stability in weighted norms for standard finite element
approximations under the quasi-uniformity assumption on the family of meshes.

\end{abstract}

\maketitle

\section{Introduction}
\label{intro}

This paper is motivated by the analysis of numerical approximations of elliptic problems
with singular sources.
The standard finite element analysis is based on the variational
formulation in Sobolev spaces. For example, for the classic Poisson problem in a bounded
domain $\Omega\in\R^n$,
it is known that the problem is well-posed in $H_0^1(\Omega)$ whenever the right hand
side is in the dual space $H^{-1}(\Omega)$.

However, the finite element method can be applied in many situations where the right hand side
is not in $H^{-1}(\O)$, and consequently, the solution is not in $H_0^1(\Omega)$.
Interesting examples of this situation arise when the right hand side is given by a
singular measure $\mu$.

Given a bounded domain $\Omega\subset\R^n$, $n=2$ or $n=3$, we consider the Poisson problem
\begin{equation}
\label{problem}
\left\{
\begin{aligned}
-\Delta u = \mu &\ \ \mbox{in}\ \ \Omega,\\
u = 0  & \ \ \mbox{on}\ \ \partial\Omega
\end{aligned}
\right.
\end{equation}

To perform a variational analysis, suitable in particular for finite element approximations,
it is natural to work with weighted Sobolev spaces. This approach has been used in several papers (see for example
\cite{AGM_delta,ABSV_delta,D_1d3d,DQ_1d3d}).

Associated with a locally integrable function $w\ge 0$ we define
the space $L^p_w(\O)$ as the usual $L^p$ space with measure $w(x) dx$ and
$\bL^p_w(\O)=L^p_w(\O)^n$.
We will also work with the Sobolev spaces
$W_w^{1,p}(\O)=\{v\in L^p_w(\O): |\nabla v|\in L^p_w(\O)\}$
and $W_{w,0}^{1,p}(\O) = \overline{C_0^\infty(\O)}$.
As it is usual, we replace $W^{1,2}$ by $H^1$.

Consider, for example, the
simple situation where $\mu$ is the Dirac $\delta$ and $0\in\O$. In this case,
$$
|\nabla u(x)|\sim |x|^{1-n}\notin L^2(\O)
$$
but
$$
|\nabla u|\in L_w^2(\O)
$$
if $w(x)=|x|^\alpha$ with $\alpha>n-2$. Therefore, to analyze this problem one can work
with a Sobolev space associated with $w$. More generally, in \cite{DQ_1d3d}
the authors consider an application which leads to a problem like \eqref{problem}
with a measure $\mu$ supported in a curve $\Gamma$ contained in
a three dimensional $\O$. They propose to work with
$w(x)=dist(x,\Gamma)^\alpha$, $0<\alpha<1$, and prove the well-posedness
of the problem in the associated weighted Sobolev space when $\alpha$ is small enough.
Afterwards, in \cite{D_1d3d},
the author gives a more general stability result for the continuous as well as
for the discrete problem obtained by the standard finite element method.
However, his proof is not correct.
Indeed, the argument given in that paper
is based on a Helmholtz decomposition in weighted spaces.
The author introduces a saddle point formulation of the problem and tries to prove the
usual inf-sup conditions that imply the existence and uniqueness of solution.
The flaw lies on the fact that (using the notation of that paper) the inf-sup conditions
needed are
$$
\sup_{\bt\neq 0} \frac{a(\bs,\bt)}{\|\bt\|_{\bK_2}}\ge\alpha_1\|\bs\|_{\bK_1}\, ,
\sup_{\bs\neq 0} \frac{a(\bs,\bt)}{\|\bs\|_{\bK_1}}\ge\alpha_2\|\bt\|_{\bK_2}
$$
where $\bK_i=\{\bs : b_i(w,\bs)=0\}$ and not those proved in \cite{D_1d3d}
where these inequalities are proved but with $\bK_i$ replaced by $\bM_i$ (see Lemma 2.1 in that paper).

Recall that to obtain a Helmholtz decomposition for a vector field $\bq$ one has
to solve
\begin{equation}
\label{auxproblem}
\left\{
\begin{aligned}
-\Delta u &= \mbox{div}\bq &\ \ \mbox{in}\ \ \Omega,\\
u &= 0  & \ \ \ \ \mbox{on}\ \ \partial\Omega
\end{aligned}
\right.
\end{equation}
with a control of $\nabla u$ in terms of $\bq$. For example, for $\bq\in L^p_w(\O)$,
we want to have the weighted a priori estimate
\begin{equation}
\label{a priori}
\|\nabla u\|_{\bL^p_w}\le C\|\bq\|_{\bL^p_w}.
\end{equation}
The first goal
of our paper is to prove these estimates for convex polygonal or polyhedral domains and for $w\in A_p$, $1<p<\infty$
(see Section \ref{continuous case}
for the definition of the Muckenhoupt classes $A_p$).
This kind of domains are very important in finite element applications.
Analogous estimates have been proved in
\cite[Theorem 2.5]{BDS} for the case of $C^1$-domains.

For non-smooth domains the convexity assumption is necessary
as it is
shown by the following example.
Consider a polygonal domain with an interior angle $\omega>\pi$ at the origin.
It is known (see \cite{Grisvard_PDE}) that the solution $u$ can have a singularity such that
$|\nabla u| \sim |x|^{s-1}$, with $s=\pi/\omega<1$, even
if the right hand side is very smooth. In such a case
$|\nabla u|^p|x|^\alpha \sim |x|^{ps-p+\alpha}$,
but $|x|^{\alpha}\in A_p$ for $-2<\alpha<2(p-1)$ (see, for example, \cite{D2}) and
$|\nabla u|\notin L^p_{|x|^\alpha}$ whenever $-2<\alpha\le -2+p(1-s)$.
On the other hand, assuming that the weight
singularities are far from the boundary, as it is the case of the model problem
considered in \cite{D_1d3d}, the weighted a priori estimates can be
generalized for non-convex Lipschitz polytopes (see \cite{OS}).

As we mentioned at the beginning, our main motivation comes from the analysis of finite
element methods. Usually, singular problems require
the use of appropriate adapted meshes to obtain good numerical approximations efficiently.
One way to produce this kind of meshes is based on the use of a posteriori error estimators.
As it is known, efficient and reliable estimators
can be derived by using the stability of the continuous problem. Therefore, these kind of results
could be obtained using \eqref{a priori}.
This was done for the case of $\mu=\delta$ in \cite{AGM_delta}.

Another way to produce adapted meshes in problems where the
location of the singularities is known a priori, like those considered here, is
by using stability results in order to
bound the approximation error by an interpolation one and then designing the meshes
in such a way that this last error is of optimal order (see, for example,
\cite{ABSV_delta} and \cite{O_pointsource}).

To prove stability results in weighted norms for general meshes
seems to be a very difficult task. Indeed, the problem is closely related
with stability in $W^{1,p}$ norms for $2<p\le\infty$, a problem that has received great attention
by people working in the theory of finite element methods in the last forty years
(see, for example, the books \cite{Ciarlet, BS_FEM} and references therein).
More precisely, as a consequence of a celebrated Rubio de Francia's
extrapolation theorem, stability in $H^1_w$ for all $w\in A_1$, would imply stability in $W^{1,p}$
for $2<p<\infty$ as well as almost stability (i.e. up to a logarithmic factor) in $W^{1,\infty}$.
As far as we know, this kind of results have not been proved for general meshes (not even assuming
regularity of the family of triangulations).

The second goal of our paper is to prove stability results in weighted norms
for standard finite element approximations under the assumption that the family
of meshes is quasi-uniform. Although this is a severe restriction for the problems
considered here, our result seems to be the first one on stability for a general
family of weights, including those given by appropriate powers of the distance to a closed
subset $\Gamma\subset\bar{\O}$ arising in the analysis of these problems. Further research
is needed to improve the results in order to allow more realistic meshes.
Our proof of the stability results make use of an
estimate proved by Rannacher and Scott \cite{RS}. Roughly speaking, their result
says that, if $u_h$ denotes the finite element approximation to the solution
$u$ of a regular problem then, for any $z\in\O$, the value $|\nabla u_h(z)|$
is bounded by a local contribution given by the average of $|\nabla u|$ in the element
containing $z$ plus a decay estimate  which is small away from $z$.
It is interesting to remark that this is the only part of our argument where the restriction
on the meshes is needed.
It is worth noting that, since our arguments are based on estimates for the Green function, the same techniques may be
applied to more general equations provided those estimates hold true.

The rest of the paper is organized as follows. In Section \ref{continuous case} we recall the Muckenhoupt
classes and prove the well posedness of the Poisson problem in weighted Sobolev spaces for convex polygonal
or polyhedral domains. Section \ref{discrete case} deals with the stability in weighted norms for finite
element approximations.

\section{The continuous case}
\label{continuous case}
\setcounter{equation}{0}

In this section we prove the weighted a priori estimate \eqref{a priori} for  equation \eqref{auxproblem}.
We will follow the arguments given in \cite{CD_Apestimate} which are a generalization
of techniques used to prove continuity of singular integral operators.
The difference with \cite{CD_Apestimate} is that
now we are interested in bounding first derivatives when the right hand side is in a weaker
space than those considered in that paper. Therefore, we need to use different estimates
for the Green function.

As  mentioned in the introduction, our motivation comes from the analysis of finite element
approximations, and, therefore, it is important to consider polygonal or polyhedral domains.
In our proofs we will use  estimates for the Green function which, for these kinds of domains,
have been proved only for  the Poisson equation. On the other hand,
if the domain is smooth enough, the estimates for the Green function that we are going to use
are known to hold for general elliptic equations (see \cite[Theorem 3.3]{Krasovskii}) and, therefore, our results apply in that case.

A weight is a non-negative measurable function $w$ defined in $\R^n$. Let us recall that, for $1<p<\infty$,
the  Muckenhoupt $A_p$ class is defined by the condition
$$
[w]_{A_p}:=\sup_{Q}\left(\frac{1}{|Q|}\int_Q w\right)\left(\frac{1}{|Q|}\int_Q w^{-\frac{1}{p-1}}\right)^{p-1}
<\infty
$$
where the supremum is taken over all cubes $Q$.

We will make use of the Hardy-Littlewood  maximal operator defined as
$$
\cM f(x) = \sup_{Q\ni x} \frac{1}{|Q|}\int_Q|f(y)|dy,
$$
where the supremum is taken over all cubes containing $x$.
It is well known that $\cM$ is bounded in $L^p_w(\R^n)$, for $1<p<\infty$,
if and only if $w\in A_p$ (see, for example, \cite{Duoandikoetxea}).

In the next section we will also work with the $A_1$ class. Recall that a weight is
in $A_1$ if
\begin{equation}
\label{peso A1}
[w]_{A_1}:=
\sup_{x\in\R^n}\frac{\cM w(x)}{w(x)}<\infty
\end{equation}
In our proofs we will make use of the well known inclusion $A_1\subset A_2$ (see \cite{Duoandikoetxea}).
We will also need the local sharp maximal operator, namely,
$$
\cM^{\#}_\Omega f(x) = \sup_{\Omega\supset Q\ni x}\frac{1}{|Q|}\int_Q |f(y)-f_Q| dy,
$$
where now the supremum is taken over all cubes containing $x$ and contained in $\Omega$.

It is known that the solution of the Poisson problem \eqref{problem}
is given by the Green function $G(x,y)$, namely,
$$
u(x) = \int_\Omega G(x,y) \, f(y)dy.
$$

In the next lemma we state some  estimates for $G$ and its derivatives.

\begin{lemma}
If $\Omega$ is a convex polygonal or polyhedral domain, there exist positive constants $C$
and $\gamma$ such that
\begin{equation}
\label{cotaG 1}
|\partial_{y_j}G(x,y)|\le C |x-y|^{1-n}
\end{equation}
and
\begin{equation}
\label{ddGeval}
|\partial_{x_i}\partial_{y_j}G(x,y)-\partial_{x_i}\partial_{y_j}G(\bar{x},y)|\le C|x-\bar{x}|^\gamma
(|x-y|^{-n-\gamma}+|\bar{x}-y|^{-n-\gamma}).
\end{equation}
\end{lemma}
\begin{proof}
For arbitrary convex domains it is proved in \cite[Proposition 1]{Fromm} that
$|\partial_{x_j}G(x,y)|\le C |x-y|^{1-n}$. Then, \eqref{cotaG 1} follows by using the
symmetry of $G(x,y)$.

Inequality \eqref{ddGeval} is proved in \cite[equation (1.4)]{GLRS} for a convex polyhedral domain.
We give a brief proof of that estimate in the case of a convex polygon $\Omega$, which is based on \cite{GLRS}.

To simplify notation, in what follows we set
 $I := |\partial_{x_i}\partial_{y_j}G(x,y)-\partial_{x_i}\partial_{y_j}G(\bar{x},y)|$.

Let  $x^{(k)}$, $k=1,\dots, K$ be the vertices of $\Omega$. We denote
$\rho_k(x) = \textrm{dist}(x,x^{(k)})$ and $\mathcal{V}_k = B(x^{(k)},\eta)\cap\Omega$
a neighborhood of $x^{(k)}$ for some fixed $\eta$ sufficiently small,
to guarantee that $\mathcal{V}_i \cap \mathcal{V}_j=\emptyset$ whenever $i\neq j$.
If $\omega_k$ is the interior  angle on $x^{(k)}$, we take $\sigma_k = \frac{\pi}{\omega_k}$.  Observe that the convexity of $\Omega$ implies $\sigma_k>1$ for every $k$.

We will make use of the following known estimates for the derivatives of $G$:
\begin{equation}\label{ddG}
  |\partial_{x_i}\partial_{y_j}G(x,y)|\le C|x-y|^{-2}
\end{equation}
(see \cite[Proposition 1]{Fromm}). Moreover, for $x,y\in\mathcal{V}_k$, we have (see \cite{MR} and \cite{KMR_Spectral}):
\begin{equation}\label{MR1}
  |D_x^{\ga} D_y^{\gb} G(x,y)|\le C \rho_k(x)^{\gs_k-|\ga|-\gve}\rho_k(y)^{-\gs_k-|\gb|+\gve}, \quad \textrm{if}\, \rho_k(x)<\rho_k(y)/2,
\end{equation}
\begin{equation}\label{MR2}
  |D_x^{\ga} D_y^{\gb} G(x,y)|\le C \rho_k(x)^{-\gs_k-|\ga|+\gve}\rho_k(y)^{\gs_k-|\gb|-\gve}, \quad \textrm{if}\, \rho_k(x)>2\rho_k(y).
\end{equation}
Finally, if $x\in\mathcal{V}_k$ and $y\in\mathcal{V}_\ell$ for $\ell\neq k$:
\begin{equation}\label{MR3}
  |D_x^\ga D_y^\gb G(x,y)|\le C \rho_k(x)^{\sigma_k-|\ga|-\gve}\rho_\ell(y)^{\sigma_\ell-|\gb|-\gve}.
\end{equation}
We fix $\gve>0$ such that $\gs_k-1-\gve>0$ for every $k$, and $\gamma$  such that $0<\gamma<\gs_k-1-\gve$ for every $k$.

 Observe that, since the singularities lie on the corners of the domain,
 it is enough to consider $x\in B(x^{(k)},\frac{\eta}{2})\cap\Omega$. We take $M>0$ a fixed constant satisfying some
  restrictions that we shall state later, and consider three main cases.

\begin{enumerate}
\item[Case 1:] $|x-y|\le M|x-\bar{x}|$.

Applying the triangle inequality and \eqref{ddG} we obtain
\begin{align*}
I \le C\Big( |\partial_{x_i}\partial_{y_j}G(x,y)|+|\partial_{x_i}\partial_{y_j}G(\bar{x},y)| \Big)
\le C\big(|x-y|^{-2}+|\bar{x}-y|^{-2}\big),
\end{align*}
and \eqref{ddGeval} follows, since  $|x-y|<M|x-\bar x|$, and $|\bar x-y|<(M+1)|x-\bar x|$.

\item[Case 2:] $|x-y|>M|x-\bar{x}|>\rho_k(x)$.

We have that
$\rho_k(\bar{x})\le \rho_k(x)+|x-\bar{x}|<(1+M)|x-\bar{x}|<(1+\frac{1}{M})|x-y|$.

Observe that, since $\rho_k(\bar x)< |x-\bar x|+ \frac{\eta}{2}< \frac{diam(\O)}{M}+ \frac{\eta}{2}$, taking $M$
sufficiently large we may assume that  $\bar{x}\in \mathcal{V}_k$.

If $y\notin\mathcal{V}_k$, we begin considering the case $y\in\mathcal{V}_\ell$ for some $\ell\neq k$. Then, applying the triangle inequality and \eqref{MR3}, and recalling that $0<\gamma<\sigma_k-1-\gve$, we obtain
\begin{align*}
  I &\le C \rho_\ell(y)^{\sigma_\ell-1-\gve} \big(\rho_k(x)^{\sigma_k-1-\gve}+\rho_k(\bar{x})^{\sigma_k-1-\gve}\big)\\ &\le C|x-\bar{x}|^{\sigma_k-1-\gve}\le C|x-\bar{x}|^\gamma,
\end{align*}
and the result follows recalling that $|x-y|>\frac{\eta}{2}$. Observe that the factor depending on $\rho_\ell(y)$ is directly bounded by a constant, since $\gs_\ell-1-\gve>0$. If $y$ is far from all the corners, this factor becomes constant and the same estimate holds.

If $y\in\mathcal{V}_k$ we separate the analysis in three subcases:
\begin{itemize}
  \item If $\rho_k(x)<\rho_k(y)/4$, we have that $|x-y|<\rho_k(x)+\rho_k(y)<\frac{5}{4}\rho_k(y)$, and $\rho_k(\bar{x})<\big(\frac{1}{4}+\frac{5}{4M})\rho_k(y).$
  We take $M>5$, so $\rho_k(\bar{x})<\rho_k(y)/2$ (therefore, \eqref{MR1} holds for $\bar{x}$). Finally, observe that $|x-\bar{x}|\le \rho_k(x)+\rho_k(\bar{x})\le \frac{3}{4}\rho_k(y)$. Then, applying \eqref{MR1} we obtain
  \begin{align*}
    I &\le C \rho_k(y)^{-\gs_k-1+\gve}\big(\rho_k(x)^{\gs_k-1-\gve}+\rho_k(\bar{x})^{\gs_k-1-\gve})\\
     &\le C\rho_k(y)^{-2-\gamma}\rho_k(y)^{\gamma-(\gs_k-1-\gve)}|x-\bar{x}|^{\gs_k-1-\gve} \\
     &\le C|x-y|^{-2-\gamma}|x-\bar{x}|^{\gamma-(\gs_k-1-\gve)}|x-\bar{x}|^{\gs_k-1-\gve} \\
     &= C|x-y|^{-2-\gamma}|x-\bar{x}|^{\gamma}.
  \end{align*}
  \item If $\rho_k(x)>4\rho_k(y)$, we have that $|x-y|\le \rho_k(x)+\rho_k(y)\le \frac{5}{4}\rho_k(x)$, and
  $\rho_k(x)\le|x-\bar{x}|+\rho_k(\bar{x})\le \frac{1}{M}|x-y|+\rho_k(\bar{x})\le \frac{5}{4M}\rho_k(x)+\rho_k(\bar{x})$.
  Combining these two estimates we obtain $\rho_k(y)\le \frac{M}{4M-5}\rho_k(\bar{x})$, which implies that \eqref{MR2} holds for $\bar{x}$ (provided $M>\frac52$). Hence, we have
  \begin{align*}
    I&\le C\rho_k(y)^{\gs_k-1-\gve}\big(\rho_k(x)^{-\gs_k-1+\gve}+\rho_k(\bar{x})^{-\gs_k-1+\gve}\big)\le C\rho_k(x)^{-2} \\
    &\le C \rho_k(x)^{-2-\gamma}\rho_k(x)^\gamma\le |x-y|^{-2-\gamma}|x-\bar{x}|^\gamma.
  \end{align*}
  \item If $\rho_k(y)/4\le\rho_k(x)\le4\rho_k(y)$ we have that $|x-y|\le\rho_k(x)+\rho_k(y)\le 5\rho_k(x)\le 5M|x-\bar{x}|$. Thanks to \eqref{ddG},
  \begin{align*}
    I &\le C\big(|x-y|^{-2}+|\bar{x}-y|^{-2}\big) \le C|x-y|^{-2-\gamma}|x-y|^{\gamma}\\ &\le C|x-y|^{-2-\gamma}|x-\bar{x}|^{\gamma}.
  \end{align*}
\end{itemize}

\item[Case 3:] $|x-y|>M|x-\bar{x}|$ and $\rho_k(x)>M|x-\bar{x}|$.

We use a mean value argument, obtaining
\begin{equation}\label{eqptomedio}
I \le |x-\bar{x}||\nabla_x\partial_{x_i}\partial_{y_j}G(z,y)|
\end{equation}
for $z = x+s(\bar{x}-x)$, $0\le s\le 1$.
Moreover, we have that
  $|x-\bar{x}|\le \rho_k(x) \le \rho_k(z)+|z-x|\le \rho_k(z)+|x-\bar{x}|$
and, consequently, $|x-\bar{x}|\le \frac{1}{M-1}\rho_k(z).$ We also have that
$\rho_k(z)\le \rho_k(x)+|x-z|\le \rho_k(x)+|x-\bar{x}|\le (1+M^{-1})|x-y|.$

As in the previous case, if $y\notin\mathcal{V}_k$, we can assume that $y\in\mathcal{V}_\ell$ for some $\ell\neq k$. In this case \eqref{MR3} gives
\begin{align*}
  I&\le C|x-\bar{x}|\rho_k(z)^{\sigma_k-2-\gve}\rho_\ell(y)^{\sigma_\ell-1-\gve}\le C|x-\bar{x}|\rho_k(z)^{\gs_k-2-\gve} \\
  &\le C|x-\bar{x}||x-\bar{x}|^{\gs_k-2-\gve}= C|x-\bar{x}|^{\gamma}
\end{align*}
and the result follows, since $|x-y|>\frac{\eta}{2}$.

Once again, if $y\in\mathcal{V}_k$, we split the proof in three subcases:
\begin{itemize}
  \item If $\rho_k(x)<\rho_k(y)/4$, recall that $|x-y|<\frac{5}{4}\rho_k(y)$, and $\rho_k(z)<\big(\frac{1}{4}+\frac{5}{4M})\rho_k(y).$
 Then, applying \eqref{MR1} we obtain
  \begin{align*}
    I &\le C |x-\bar{x}|\rho_k(z)^{\gs_k-2-\gve}\rho_k(y)^{-\gs_k-1+\gve} \\
     &\le C|x-\bar{x}|^\gamma \rho_k(z)^{\gs_k-1-\gve-\gamma}\rho_k(y)^{-\gs_k-1+\gve} \\
     &\le C|x-\bar{x}|^\gamma \rho_k(y)^{-2-\gamma}\le C|x-\bar{x}|^\gamma|x-y|^{-2-\gamma}.
  \end{align*}
  \item If $\rho_k(x)>4\rho_k(y)$, we have  $|x-y|\le\frac{5}{4}\rho_k(x)\le \frac{5}{4}\big(\rho_k(z)+|x-z|\big)\le \frac{5M}{4(M-1)}\rho_k(z)$, and
  $\rho_k(y)\le \frac{M}{4(M-1)}\rho_k(z)$. Applying \eqref{MR2} we have
  \begin{align*}
    I&\le C|x-\bar{x}|\rho_k(z)^{-\gs_k-2+\gve}\rho_k(y)^{\gs_k-1-\gve}\\
    &\le C|x-\bar{x}|^\gamma \rho_k(z)^{-\gs_k-1+\gve-\gamma}\rho_k(y)^{\gs_k-1-\gve}\\
    &\le C|x-\bar{x}|^\gamma \rho_k(z)^{-1-\gamma}\le C|x-\bar{x}|^\gamma|x-y|^{-2-\gamma}.
  \end{align*}
  \item If $\rho_k(y)/4\le\rho_k(x)\le4\rho_k(y)$ we have that $|x-y|\le|x-z|+|z-y|\le \frac{1}{M-1}|x-y|+|z-y|$, which leads to $|x-y|\le\frac{M-1}{M}|z-y|$. Here we need an additional estimate given in \cite{MR}, namely,
  $$|D_x^{\alpha}D_y^{\beta}G(z,y)|\le C|x-y|^{-|\alpha|-|\beta|}$$
  that holds if  $\rho_k(y)/5<\rho_k(z)<5\rho_k(y)$ (valid for $z$ if $M>5$). Therefore, we have
  \begin{align*}
    I &\le C|x-\bar{x}||z-y|^{-3}\le C|x-\bar{x}|^{\gamma}|z-y|^{-2-\gamma}
  \end{align*}
 and the result follows.
\end{itemize}
\end{enumerate}
\end{proof}

Given $w\in A_p$ we consider Problem \ref{auxproblem} with $\bq\in \bL^p_w(\Omega)$. Recalling that $G(x,0)=0$, we have
\begin{equation}
\label{representacion}
u(x) = \int_\Omega G(x,y) \, \mbox{div}\bq(y)dy
= -\int_\Omega \nabla_yG(x,y)\cdot\bq(y)dy.
\end{equation}
Here the first integral has to be understood in a weak sense while
the second one is well-defined. Indeed, it is not difficult to see
that $L^p_w(\Omega)\subset L^1(\Omega)$. Therefore, that $\nabla_yG(x,y)\cdot\bq(y)$
is an integrable function follows from  estimate (\ref{cotaG 1}).

We will use the following unweighted known a priori estimate.

\begin{lemma}
Let  $\Omega$ be a convex domain and $u$ be the solution of problem
\eqref{auxproblem} then, for  $1<p<\infty$, the following estimate holds,
\begin{equation}
\label{sin peso}
\|\nabla u\|_{\bL^p}\le C \|\bq\|_{\bL^p}
\end{equation}
\end{lemma}
\begin{proof}
See \cite{Fromm}.
\end{proof}

The argument given in \cite{CD_Apestimate} makes use of the following inequality
proved in \cite{DRS}.

\begin{lemma}
    \label{L:f-fOmega}
    For $f\in L^1_{loc}(\O)$, $w\in A_p$ and $f_\O$ the mean value of $f$ over $\O$, we have
    $$
    \|f-f_\O\|_{L^p_{w}}\le C \|\cM^{\#}_{\Omega}f\|_{L^p_{w}}.
    $$
\end{lemma}

In what follows we make use of the fact that $C_0^\infty(\Omega)$ is dense in $L^p_w(\Omega)$ and, therefore, we can assume
that $\bq$ is smooth. Hence, pointwise values of the derivatives of $u$ are well-defined.

\begin{lemma}
\label{L:Msharp}
Let  $\Omega$ be a convex polygonal or polyhedral domain and $u$ be the
solution of problem \eqref{auxproblem}. Then, for any $s>1$, we have
$$
\cM^\#_\Omega(|\nabla u|)(\bar{x}) \le C (\cM|\bq|^s)^\frac{1}{s}(\bar{x})
$$
for all $\bar{x}\in\Omega$.
\end{lemma}
\begin{proof}
We extend $\bq$ by zero outside $\Omega$. Given $\bar{x}\in\Omega$, let $Q\subset\Omega$ be a cube such that
$\bar{x}\in Q$ and let $Q^*$ be
an expansion of $Q$ by a factor $2$. We decompose $\bq = \bq_1+\bq_2$, where $\bq_1 = \chi_{Q^*}\bq$, and call $u_i$
the solution of \eqref{auxproblem} with right hand side given by $\mbox{div}\bq_i$.
It is enough to bound the sharp maximal function of $\partial_{x_i}u$ for any $i$.

It is known that to estimate $\cM^\#_\Omega$, one can replace the average $f_Q$ by any constant $a$.
We take $a=\partial_{x_i}u_2(\bar{x})$, hence,
\begin{align*}
\frac{1}{|Q|}&\int_Q |\partial_{x_i}u(x)-\partial_{x_i}u_2(\bar{x})|dx \\
&\le \frac{1}{|Q|}\int_Q|\partial_{x_i}u_1(x)| dx + \frac{1}{|Q|}
\int_Q |\partial_{x_i}u_2(x)- \partial_{x_i}u_2(\bar{x})| dx =: (i) + (ii).
\end{align*}
Given $s>1$, using  H\"older's inequality, the unweighted estimate (\ref{sin peso}) for
$L^s$, and recalling that $\bq_1$ vanishes outside $\Omega\cap Q^*$, we have
\begin{align*}
(i) \le \left(\frac{1}{|Q|}\int_Q |\partial_{x_i} u_1(x)|^s dx\right)^{\frac{1}{s}}
\le C\left(\frac{1}{|Q|}\int_{Q^*} |\bq_1(x)|^s dx\right)^\frac{1}{s}
\le C(\cM|\bq|^s)^{\frac{1}{s}}(\bar{x}).
\end{align*}

To bound $(ii)$, since $x$ and $\bar{x}$ are outside the support of $\bq_2$, we can take the derivative inside the integral
in the expression for $u_2$ given by \eqref{representacion}, and using (\ref{ddGeval}), we obtain
\begin{align*}
(ii)
&\le\frac{1}{|Q|}\int_Q\int_{\Omega\cap(Q^*)^c}|\partial_{x_i}\nabla_y G(x,y)
-\partial_{x_i}\nabla_y G(\bar{x},y)| |\bq_2(y)| dy dx \\
&\le\frac{C}{|Q|}\int_Q\int_{(Q^*)^c}|x-\bar{x}|^\gamma(|x-y|^{-n-\gamma}+|\bar{x}-y|^{-n-\gamma})|\bq(y)| dy dx
\end{align*}
Now, since $x,\bar{x}\in Q$ and $y\in (Q^*)^c$, we have $|x-y|\sim |\bar{x}-y|\ge\frac{\ell(Q)}{2}$,
where $\ell(Q)$ denotes the length of the edges of $Q$,
and therefore,
$$
(ii) \le C\frac{\ell(Q)^\gamma}{|Q|}\int_Q\int_{(Q^*)^c}\frac{|\bq(y)|}{|\bar{x}-y|^{n+\gamma}} dy dx
\le C\int_{\ell(Q)/2<|\bar{x}-y|}\frac{\ell(Q)^{\gamma}|\bq(y)|}{|\bar{x}-y|^{n+\gamma}} dy
\le C \cM|\bq|(\bar{x}).
$$
where the last inequality follows in a standard way (see \cite[Page 506]{Hedberg}).

But, by  H\"older's inequality,
$\cM|\bq|(\bar{x})\le(\cM|\bq|^s)^\frac{1}{s}(\bar{x})$
and so the lemma is proved. \end{proof}

Now, we are able to prove our main result, namely, the weighted estimate for $\nabla u$.

\begin{thm}
\label{teo principal}
Let  $\O$ be a convex polygonal or polyhedral domain.
Given $1<p<\infty$ and $w\in A_p$, if $\bq\in\bL^p_w(\O)$
and $u$ is the solution of Problem \eqref{auxproblem},
there exists
a constant $C$ depending on $p$, $\O$ and $w$ such that,
$$
\|\nabla u\|_{\bL^p_w}\le C\|\bq\|_{\bL^p_w}.
$$
\end{thm}
\begin{proof}
We have
$$
\|\nabla u\|_{\bL^p_w}
\le \|\nabla u-(\nabla u)_\O\|_{\bL^p_w}
+ \|(\nabla u)_\O\|_{\bL^p_w}
=: I + II.
$$
Now, it is known that if $w\in A_p$ then $w\in A_{\frac{p}{s}}$ for some $s$ such that $1<s<p$
(see, for example, \cite[Corollary 7.6]{Duoandikoetxea}).
Then, using Lemmas \ref{L:f-fOmega} and \ref{L:Msharp}, and that $\cM$ is bounded on $L^{\frac{p}{s}}_w$,
we obtain
$$
I\le C \|\cM^\#_\Omega(|\nabla u|)\|_{L^p_w}
 \le C \|\c(M|\bq|^s)^\frac{1}{s}\|_{L^p_w}
 \le C \|\bq\|_{\bL^p_w}
$$
Then, to finish the proof it is enough to bound $|(\nabla u)_\Omega|$. Using  H\"older's inequality,
with exponent $s$ in the first inequality and with exponent $p/s$ in the third one, and the a priori estimate (\ref{sin peso})
for the second inequality, we obtain
\begin{align*}
|(\nabla u)_\Omega|
& \le \left(\frac{1}{|\Omega|}\int_\Omega |\nabla  u(x)|^s dx\right)^\frac{1}{s}
\le C\left(\frac{1}{|\Omega|}\int_\Omega|\bq|^s\right)^\frac{1}{s}\\
&\le C\left(\frac{1}{|\Omega|}\int_\Omega |\bq|^p w(x)dx\right)^\frac{1}{p}
\left(\frac{1}{|\Omega|}\int_{\Omega}w(x)^{-\frac{s}{p-s}}\right)^\frac{p-s}{ps}
\end{align*}
and the last integral is finite since $w\in A_{\frac{p}{s}}$.
\end{proof}

Now we can prove the well-posedness of problem \eqref{problem}. This result follows from
Theorem \ref{teo principal} by standard functional analysis arguments.
We give it here for the sake of completeness.
In the proof we will use the following weighted Poincar\'e inequality (see \cite[Ch.2 Section 15]{KO_Hardy}).
If $w\in A_p$ then there exists a constant $C$ such that
\begin{equation}
\label{poincare}
\|v\|_{L^p_w}\le C \|\nabla v\|_{\bL^p_w} \qquad \forall v\in W_{w,0}^{1,p}(\O)
\end{equation}
Given $w\in A_p$ we intrduce its dual weight $w':=w^{-1/(p-1)}$. It is known that
$w'\in A_{p'}$ (see for example \cite{D2}) and that $\bL_w^p(\O)'=\bL^{p'}_{w'}(\O)$.

\begin{corollary}
\label{T:wellposedness}
If $\Omega$ is a convex polygonal or polyhedral domain,
$1<p<\infty$ and $w\in A_p$ then, given $\mu\in(W_{w',0}^{1,p'}(\O))'$
there exists a unique solution of Problem \eqref{problem} satisfying,
\begin{equation}
\label{a priori para u}
\|u\|_{W_w^{1,p}}
\le C\|\mu\|_{(W_{w',0}^{1,p'})'}
\end{equation}
\end{corollary}
\begin{proof}
Set ${\mathcal L}(\nabla v):=-\la \mu,v \ra$. Using the Poincar\'e inequality \eqref{poincare}
we have $|{\mathcal L}(\nabla v)|\le \|\mu\|_{(W_{w',0}^{1,p'})'}\|\nabla v\|_{\bL^{p'}_{w'}}$.
Therefore, ${\mathcal L}$ defines a continuous linear functional
on the gradient fields of functions in $W_{w',0}^{1,p'}$ and so,
by the Hahn-Banach theorem, it can be
extended to all $\bL^{p'}_{w'}$. Therefore, there exists $\bq\in\bL^{p}_{w}$ such that
$\|\bq\|_{\bL^{p}_{w}}=\|\mu\|_{(W_{w',0}^{1,p'})'}$ and $\la \bq,\nabla v \ra=-\la \mu,v \ra$. Then
$\mbox{div}\bq=\mu$, and therefore, the existence of $u$ and
the estimate \eqref{a priori para u}
are immediate consequences of Theorem \ref{teo principal} and \eqref{poincare}.
\end{proof}

The results obtained above can be applied to the problem considered in \cite{D_1d3d}.
In that paper the author considers a problem like \eqref{problem} with $\mu$ supported
in a curve contained in a three dimensional domain. He works with a weighted space where the
weight is a power of the distance to the curve.
More generally one can consider
$\Gamma\subset\overline{\O}\subset\R^n$ where $\Gamma$ is a compact set.
We will assume that $\Gamma$ is a $k$-regular set for some $0\le k<n$, namely, there exist constants $C_1,C_2>0$ such that
$C_1 r^k\le {\mathcal H}^k(B(x,r)\cap \Gamma)\le C_2 r^k$ for every
$x\in \Gamma$ and $0<r\le diam(\Gamma)$, where ${\mathcal H}^k$ denotes
the $k$-dimensional Hausdorff measure. Let us remark that $k$ is not necessarily an integer. However,
if $\Gamma$ is smooth then $k$ is the usual dimension.

To simplify notation we introduce $w_\lambda:=dist(x,\Gamma)^{\lambda}$.
It is known that, if $\Gamma$ is a $k$-regular set, then, for $1\le p<\infty$,
\begin{equation}
\label{potencias en Ap}
-(n-k)<\lambda<(n-k)(p-1)
\Longrightarrow
w_\lambda\in A_p
\end{equation}
(see \cite[Lemma 2.3,vi]{FS} or \cite[Appendix B]{AD_Divergence})

\begin{thm}
\label{aplicacion en el continuo}
If $\O$ is a convex polygonal or polyhedral domain,
$\Gamma\subset\overline{\O}$ is a $k$-regular set and $1<p<\infty$ then,
for $-(n-k)<\lambda<(n-k)(p-1)$, given
$\mu\in (W_{w_{-\lambda/(p-1)},0}^{1,p'}(\O))'$
there exists a unique solution $u\in W_{w_\lambda}^{1,p}(\O)$ of Problem \ref{problem} satisfying,
$$
\|u\|_{W_{w_\lambda}^{1,p}}
\le C\|\mu\|_{\left(W_{w_{-\lambda/(p-1)},0}^{1,p'}\right)'}
$$
\end{thm}
\begin{proof}
In view of \eqref{potencias en Ap} the result is an immediate consequence
of
Corollary \ref{T:wellposedness}.
\end{proof}

In particular, taking $n=3$, $k=1$ and $p=2$ we obtain the result stated
in \cite[Corollary 2.2]{D_1d3d}.

\section{The discrete case}
\label{discrete case}
\setcounter{equation}{0}

The goal of this section is to prove weighted stability estimates
for finite element approximations of the Poisson equation.

Given a convex polygonal or polyhedral domain $\Omega$ and a family of triangulations
${\mathcal T}_h$, where as usual $h>0$ denotes the maximum of the diameters of
the elements, let $V^k_h$ be the space of continuous piecewise
polynomial functions of degree $k\ge 1$.
The finite element approximation $u_h^k\in V^k_h$ of $u$ is given by
$$
\int_\Omega \nabla u^k_h\cdot\nabla v
=\int_\Omega \nabla u\cdot\nabla v
\qquad\forall v\in V^k_h
$$
Observe that $u^k_h$ is well defined for any $u\in W^{1,1}(\Omega)$, in particular,
for any $u\in L^1(\Omega)$ such that $\nabla u\in\bL^p_w$ for some $w\in A_p$.

Since $k$ will be fixed, we will drop it from now on and will write simply $V_h$ and $u_h$. Also, as in the continuous case,
by density we may assume that $u$ is smooth.

\begin{thm}
\label{estabilidad con pesos}
Let $\Omega$ be a convex polygonal or polyhedral domain and assume that the family of partitions ${\mathcal T}_h$
is quasi-uniform. If $w\in A_1$ then there exists a constant $C$, depending only on $[w]_{A_1}$, such that
$$
\|\nabla u_h\|_{\bL^2_{w}(\Omega)}\le C\|\nabla u\|_{\bL^2_{w}(\Omega)}.
$$
\end{thm}
\begin{proof}
Under the quasi-uniformity assumption it was proved in
\cite[Corollary 8.2.8, Page 219]{BS_FEM} that there exist positive constants
$C$ and $\lambda$ such that, if $T_z$ is a triangle containing $z$ then,

\begin{equation}
\label{localizacion}
\begin{aligned}
|\nabla u_h(z)|^2
\le C &\left\{\left(\frac1{h^n}\int_{T_z}|\nabla u(x)| dx\right)^2\right.\\
&\left.+\int_\Omega\frac{h^\lambda}{(|x-z|^2+h^2)^{\frac{n+\lambda}2}}|\nabla u(x)|^2 dx\right\},
\end{aligned}
\end{equation}

Actually, \cite[Corollary 8.2.8, Page 219]{BS_FEM} needs a more restrictive condition on the angles in the 3D-case. However,
the result is still true for general convex polyhedral domains. Indeed, this can be seen with a slight modification of the
arguments in \cite{GLRS} (see \cite{O_pointsource}).

Therefore,
$$
|\nabla u_h(z)|^2
\le C \left\{\cM(|\nabla u(z)|)^2
+\int_\Omega\frac{h^\lambda}{(|x-z|^2+h^2)^{\frac{n+\lambda}2}}|\nabla u(x)|^2 dx\right\}.
$$
Then, multiplying by $w(z)$ an integrating we obtain
\begin{equation}
\label{localizacion 2}
\begin{aligned}
\int_\Omega|\nabla u_h(z)|^2w(z)dz
\le C &\left\{\int_\Omega\cM(|\nabla u(z)|)^2 w(z)dz\right.\\
&\left.+\int_\Omega\int_\Omega\frac{h^\lambda|\nabla u(x)|^2 w(z)}{(|x-z|^2+h^2)^{\frac{n+\lambda}2}}dx dz\right\}.
\end{aligned}
\end{equation}
But
$$
\int_\Omega\frac{h^\lambda w(z)}{(|x-z|^2+h^2)^{\frac{n+\lambda}2}} dz
\le \frac1{h^n}\int_{|x-z|\le h} w(z)dz + \int_{|x-z|> h} \frac{h^\lambda w(z)}{|x-z|^{n+\lambda}}dz
$$
Using a well-known
argument (see \cite[Page 506]{Hedberg}) to bound the second term we have
$$
\int_\Omega\frac{h^\lambda w(z)}{(|x-z|^2+h^2)^{\frac{n+\lambda}2}} dz
\le C\cM w(x)
$$
and, therefore, interchanging the order of integration in \eqref{localizacion 2} we
obtain
$$
\int_\Omega|\nabla u_h(z)|^2 w(z)dz
\le C \left\{\int_\Omega\cM(|\nabla u(z)|)^2 w(z)dz
+\int_\Omega |\nabla u(x)|^2 \cM w(x)dx\right\}.
$$
In particular, if $w\in A_1$, using \eqref{peso A1} and recalling that $A_1\subset A_2$
and so the maximal operator is bounded in $L^2_w$, we conclude that
$$
\|\nabla u_h\|_{\bL^2_{w}(\Omega)}\le C\|\nabla u\|_{\bL^2_{w}(\Omega)}.
$$
\end{proof}
From a known extrapolation theorem we obtain the following result.
\begin{corollary}
Under the hypotheses of the previous theorem, for $2<p<\infty$ there exists a constant $C$
depending only on $p$ and $[w]_{A_1}$, such that,
$$
\|\nabla u_h\|_{\bL^p_{w}(\Omega)}\le C\|\nabla u\|_{\bL^p_{w}(\Omega)}.
$$
\end{corollary}
\begin{proof} See \cite[Corollary 3.5, Page 30]{D2}.
\end{proof}

Next, using a standard duality argument combined with the weighted a priori estimates given in the
previous section, we extend the stability result for weights with inverse in $A_1$.
\begin{corollary}
\label{dualidad}
Under the hypotheses of the previous theorem,
if $w^{-1}\in A_1$ then,
$$
\|\nabla u_h\|_{\bL^2_{w}(\Omega)}\le C\|\nabla u\|_{\bL^2_{w}(\Omega)}.
$$
\end{corollary}
\begin{proof}
Take $\bq=w\nabla u_h$ and let $v$ be the solution of $-\Delta v=\mbox{div}\bq$
vanishing on $\partial\Omega$. From Theorems \ref{estabilidad con pesos} and \ref{teo principal} we know that
$$
\|\nabla v_h\|_{\bL^2_{w^{-1}}}\le C\|\nabla v\|_{\bL^2_{w^{-1}}}\le C\|\bq\|_{\bL^2_{w^{-1}}}.
$$
Then
$$
\begin{aligned}
&\|\nabla u_h\|^2_{\bL^2_{w}}
=\int_\Omega\nabla u_h\cdot\bq
=\int_\Omega\nabla u_h\cdot\nabla v
=\int_\Omega\nabla u\cdot\nabla v_h\\
&\le C\|\nabla u\|_{\bL^2_{w}}
\|\nabla v_h\|_{\bL^2_{w^{-1}}}
\le C \|\nabla u\|_{\bL^2_{w}}
\|\bq\|_{\bL^2_{w^{-1}}}
=C \|\nabla u\|_{\bL^2_{w}}
\|\nabla u_h\|_{\bL^2_{w}}
\end{aligned}
$$
\end{proof}

As we have done in the continuous case we can apply these results to the
problem considered in \cite{D_1d3d} as well as to the generalization introduced
at the end of the previous section. With the notation used there we have

\begin{thm} Under the hypotheses of Thorem \ref{estabilidad con pesos},
if $\Gamma\subset\overline{\O}$ is a $k$-regular set then, for  $-(n-k)<\lambda<n-k$,
$$
\|\nabla u_h\|_{\bL^2_{w_\lambda}}\le C \|\nabla u\|_{\bL^2_{w_{-\lambda}}}
$$
\end{thm}
\begin{proof} It is an immediate consequence of Theorem \ref{estabilidad con pesos} and
Corollary \ref{dualidad}
because either $w_\lambda\in A_1$ or $w_{-\lambda}\in A_1$ by \eqref{potencias en Ap}.
\end{proof}

\section*{Acknowledgement}
We thank Dmitriy Leykekhman for helpful comments and references on the estimates for the Green function.

\bibliographystyle{plain}
\bibliography{bibdoc}

\end{document}